\documentclass{amsart}

\usepackage{amsfonts}

 \usepackage{amsmath}
 
 \usepackage{graphicx}
\usepackage{amssymb, amsmath, amsfonts, amsthm, mathtools,hyperref} 
\usepackage{latexsym}
\usepackage[usenames,dvipsnames]{color}

\usepackage{minitoc}

\usepackage{amsthm}

 \usepackage{geometry}

\def\cref#1{Corollary~\ref{#1}}

\newcommand{\bO}{{\bf\Omega}}

\newcommand{\C}{\mathbb C}

\newcommand{\Z}{\mathbb Z} 

\newcommand{\R}{\mathbb R}

\newcommand{\Q}{\mathbb Q}

\renewcommand{\k}{{\mathfrak{k}}{}}

\renewcommand{\k}{{\mathfrak{k}}{}}
\newcommand{\CO}{{\mathcal O}{}}

\newcommand{\CD}{\mathcal D}

\newcommand{\CH}{\mathcal H}
\newcommand{\CB}{\mathcal B}

\newcommand{\CF}{\mathcal F}

\newcommand{\vf}{\varphi}

\renewcommand{\k}{{\bf k}}

\newtheorem{thm}{Theorem}[section] 

\newtheorem{rmk}[thm]{Remark}

\newtheorem{prop}[thm]{Proposition}
\newtheorem{corr}[thm]{Corollary}
\newtheorem{lem}[thm]{Lemma}

\usepackage{fancyhdr}
\title{A locally compact quantum group arising from quantization
of the affine group of a local field}
 \date{}
 \author{David Jondreville}
 \address{Laboratoire de Math\'ematique de Reims, CNRS FRE 2011, David.Jondreville@ac-reims.fr}

\begin{document}
\maketitle
\begin{abstract}
Using methods coming from non-formal equivariant  quantization, 
we construct in this short note a unitary \emph{dual 2-cocycle} on a discrete family of quotient groups of  subgroups
of the affine group of a local field  (which is not of characteristic $2$, nor an extension of $\Q_2$). 
Using results of  De Commer about Galois objects in operator algebras, 
we obtain  new examples of  locally compact quantum groups in the setting of von Neumann algebras.
\end{abstract}


\section{Introduction}

The construction of concrete examples of  locally compact quantum groups  (LCQG) in the setting of  von Neumann algebras \cite{KV1,KV2} is still of upmost importance.
The seminal work of De Commer \cite{dC} provides  a way to construct a new LCGQ from a Galois object on a given LCQG. 
An important class of Galois objects are associated with  $2$-cocycles \cite[Section 5]{dC}. 
Recall that a   \emph{unitary dual $2$-cocycle} on a LCQG   $(\mathcal N,\Delta)$, with dual
quantum group
$(\hat{\mathcal N},\hat\Delta)$ (we use the standard notations, see e.g$.$ \cite{dC,NT}), 
is a unitary  element $F\in \hat{\mathcal N}\bar\otimes \hat{\mathcal N}$ satisfying the cocycle relation:
\begin{equation}
\label{2CC}
(1\otimes F)(\iota\otimes \hat\Delta)(F)=(F\otimes 1)(\hat\Delta\otimes\iota)(F).
\end{equation}

At the level of an ordinary locally compact group $G$, the explicit construction of a unitary dual $2$-cocycle $F$  is already  a non-trivial task. In that case,
$F$   is a unitary element of the group von Neumann algebra $W^*(G\times G)$
and the cocycle relation is equivalent
(see \cite[Section 4]{NT}) to the associativity of the following left-equivariant deformed 
product of the Fourier algebra $A(G)$:
$$
f_1\star f_2(g)= \check f_1 \otimes \check f_2\big(\lambda_g\otimes\lambda_g\, F^*\big).
$$ 
In the formula above, $\lambda$ is the left regular representation $\lambda_g(f)(g')=f(g^{-1}g')$
and $f \mapsto\check f$ is the standard identification of  $A(G)$ with  the pre-dual of $ W^*(G)$.

Hence, there is a simple strategy to construct dual $2$-cocycles on groups. One starts with a left-equivariant
and associative product $\star$ on $A(G)$ (or on a dense subalgebra of it), written in term of a distribution 
(in the sense of Bruhat 
\cite{Bruhat}) $K$ on $G\times G$  as follows:
$$
f_1\star f_2=\int_{G\times G} K(g_1,g_2)\;\rho_{g_1}(f_1)\;\rho_{g_2}(f_2)\;dg_1dg_2,
$$
where $\rho$ is right regular action $\rho_g(f)(g')=f(g'g)$ and $dg$ is a left-invariant Haar measure.
One then considers the following (formal) convolution operator:
$$
F:= \int_{G\times G} \overline{K(g_1,g_2)}\;\lambda_{g_1^{-1}}\otimes\lambda_{g_2^{-1}}\;dg_1dg_2.
$$
All what remains to do then, is to verify that this operator  makes  sense and is unitary on $L^2(G)$ (for the left-invariant Haar measure).
Now, the point is that an associative and left-equivariant product on a group $G$ can be constructed out of a non-formal, $G$-equivariant  quantization map on $G$. 
By this we mean a triple $(\bO,\CH,\alpha)$ consisting in a unitary operator $\bO$ from $L^2(G)$ to the Hilbert space 
$\mathcal L^2(\CH)$ of Hilbert-Schmidt
operators on a separable Hilbert space $\CH$, together with an action of $G$ on $\CB(\CH)$ satisfying the covariance relation:
$$
\bO\big(\lambda_g(f)\big)=\alpha_g\big(\bO(f)\big),\quad\forall f\in L^2(G).
$$
One can then define an associative and left-equivariant product on $L^2(G)$ by transporting the product of $\mathcal L^2(\CH)$ to $L^2(G)$:
$$
f_1\star f_2:=\bO^*\big(\bO(f_1)\,\bO(f_2)\big).
$$

  In \cite{GJ2}, we have constructed such a non-formal, equivariant  quantization map and the aim of this short note is to check that the 
  associated candidate for
  a unitary dual $2$-cocycle satisfies this unitary property.  
  The family of groups we consider here are quotient groups of subgroups of the affine group of a local field. To simplify a little this introductive presentation,
  let us consider the case of the field of $p$-adic numbers $\Q_p$ with $p$ an odd prime number. Denoting by $\Z_p$ the ring of integers of $\Q_p$, 
  our groups  are 
  unimodular semidirect products of the form  $(1+p^n\Z_p)\ltimes (\Q_p/p^{-n}\Z_p)$. We then have a unitary quantization map (see     \cite[Proposition 3.9]{GJ2})
  $$
  \bO: L^2\big((1+p^n\Z_p)\ltimes (\Q_p/p^{-n}\Z_p)\big)\to \mathcal L^2\big(L^2(1+p^n\Z_p)\big),
  $$
   defined for $f\in L^1\cap L^2\big((1+p^n\Z_p)\ltimes (\Q_p/p^{-n}\Z_p)\big)$ by
  $$
   \bO(f)=p^{n}\sum_{[t]\in \Q_p/p^{-n}\Z_p}\int_{1+p^n\Z_p}f(a,[t])\,\pi(a,t)\Sigma\pi^*(a,t) \,da.
  $$
  Here $\Sigma$ is the (bounded) operator on  $L^2(1+p^n\Z_p)$ given by $\Sigma\vf(a)=\vf(a^{-1})$, $\pi$ is the irreducible unitary representation of
  $(1+p^n\Z_p)\ltimes \Q_p$  (a subgroup of the affine group of $\Q_p$) given by 
  $$
  \pi(a,t)\varphi(a_0) :=\Psi( a_0^{-1}at) \,
\varphi\big(a^{-1}a_0\big),
$$
and $\Psi$ is the basic unitary character of $(\Q_p,+)$ given by
$$
\Psi(x)=e^{2i\pi\sum_{n=-k}^{-1}a_np^n}\quad \mbox{if}\quad x=\sum_{n=-k}^{\infty}a_np^n\in \Q_p,\quad a_n\in\{0,1,\dots,p-1\}.
$$
  The candidate for the unitary dual $2$-cocycle associated with this quantization map is given by:
  $$
 F:= p^{2n}\sum_{[t_1],[t_2]}\int
 \Psi\big((a_1-a_1^{-1})t_2-(a_2-a_2^{-1})t_1\big)\,
 \lambda_{(a_1,[t_1])^{-1}}\otimes  \lambda_{(a_2,[t_2])^{-1}}\, da_1da_2,
 $$
 where the sum is over $( \Q_p/p^{-n}\Z_p)^2$ and the integral is over $(1+p^n\Z_p)^2$.
 
 Our main result  here, Theorem \ref{Th},  is that $F$   indeed defines  a unitary operator,
 the $2$-cocyclicity being automatic. However, for the sake of completeness, we give a direct proof of $2$-cocyclicity in the Appendix.
 As mentioned 
 earlier, De Commer's results then give  an example of a locally compact quantum group in the von Neumann algebraic setting.
 Moreover, the results of \cite{NT}  apply too and give a deformation theory for actions of that groups on $C^*$-algebras.
  
\section{Preliminaries}

In this section, we first introduce the basic notations and tools we need. We then give a review 
of the   pseudo-differential calculus  constructed in \cite{GJ2} and extract a dual $2$-cocycle out of the associated $\star$-product
(that is, the composition law of symbols).  This pseudo-differential calculus is termed the \emph{$p$-adic Fuchs calculus}
because it imitates Unterberger's  Fuchs calculus \cite{Un84} in the non-Archimedean setting.

\subsection{Setting and tools.}

Let ${\bf k}$ be a non-Archimedean local field. In characteristic zero, $\k$ is isomorphic to   a finite 
degree extension of the field
of $p$-adic numbers $\Q_p$. In positive characteristic, $\k$ is isomorphic to
a field of Laurent 
series $\mathbb F_\ell((X))$ with coefficients in a finite field $\mathbb F_\ell$.
We let $|.|_{\bf k}$ be the ultrametric absolute value,   $\CO_{\bf k}$ be the ring of integers of $\k$ (the unit ball), $\varpi$  be a 
generator of its unique maximal ideal and $\CO_{\k}^{\times}$ be 
the multiplicative group of units (the unit sphere). 
We let $p$ be the characteristic of the residue field $\CO_{\k}/\varpi\CO_{\k}$ and $q=p^f$  its cardinality.
The uniformizer $\varpi$ then satisfies $|\varpi|_{\k} = q^{-1}$. 
For every  integer $n \geq 1$, we denote by $U_n:=1+\varpi^n\CO_{\k}\subset \CO_{\k}^{\times}$ the   group of higher principal units.
Note that $U_n$ is  an open and  compact  subgroup of $\k^\times$ and it
 acts by dilations on $\k$. We can therefore  consider the  semidirect product $G_n:=U_n  \ltimes  \k$, that we
  view as a  subgroup of  the affine group $\k^\times\ltimes\k$, 
   with group law
  \begin{equation}
\label{grouplaw}  (x,t).(x',t')  := (xx', x'^{-1}t+t')\,,\quad x,x'\in\k^\times\,,\;t,t'\in\k.
 \end{equation}
 
 The group $G_n$ will be the \emph{covariance group} of our pseudo-differential calculus.
The subgroup $U_n$  will play the role of the \emph{configuration space} and the role
of the \emph{phase space} will be played by  the 
 homogeneous space $X_n :=  G_n/H_n$, where $H_n$ is the closed subgroup 
 $  \{1\} \times \varpi^{-n}\CO_{\k}\subset G_n$. Since $H_n$ is normal  in $G_n$,   
the quotient $X_n$ is  endowed with a group structure. \emph{This is on the group $X_n$ that we will construct a unitary dual $2$-cocycle}. 
Of course, $X_n$ is  isomorphic 
to the semidirect product  $U_n \ltimes \Gamma_n$,
where $\Gamma_n := \k  /{\varpi^{-n}\CO_{\k}}$ is discrete and Abelian.
As a locally compact  group, our phase space  $X_n$ is not compact, nor 
discrete and non-Abelian. We will denote the elements of $X_n$
by pairs $(a,[t])$, where $a \in U_n$ and $[t]:=t+\varpi^{-n}\CO_\k\in\Gamma_n$,
 so that the group law of $X_n$ is given by
$$
(a,[t]).(a',[t']) \; := \; (aa', [a'^{-1}t+t']).
$$
To simplify the notations, we will frequently denote an element of $G_n$ by $g:=(a,t)$ and an element
of the quotient group $X_n$ by $[g]:=(a,[t])$.  
We  normalize the Haar measure $dt$ of $\k$
so that ${\rm Vol}(\CO_\k)=1$.  Since $U_n$ is open in $\k^\times$, the Haar measure of $U_n$ is given by  the restriction of the Haar measure of
$\k^\times$ (which is $|t|_\k^{-1}dt$) but since $U_n\subset \CO_k^\times$, it is just $dt$.
 The $G_n$-invariant measure of  $X_n$
(which is also a Haar measure for the quotient group structure),
denoted by $d[g]$, is chosen to be the product of the Haar measure 
of $U_n$ by the counting measure on $\Gamma_n$.
In particular,  both groups $G_n$ and $X_n$ are unimodular. 
 Once for all, we fix  a unitary character $\Psi$ of the additive group $\k$ 
which is required to be trivial on 
$\CO_\k$ but not on $\varpi^{-1}\CO_\k$. 
We also deduce by \cite[Theorem 23.25]{HR} that $\widehat \Gamma_n $ is isomorphic to the additive group 
$\varpi^n\CO_\k$.

Important ingredients of the $p$-adic Fuchs calculus, are the square function 
$$
\sigma:U_n\to U_n\,,\quad
a\mapsto a^2,
$$ 
and the hyperbolic sine type function
$$
\phi:U_n\to \varpi^n\CO_\k\,,\quad a \mapsto a-a^{-1}.
$$
The crucial requirement for the $p$-adic Fuchs calculus to behave well, is that  the mappings $\sigma$ and 
$\phi$ are $C^1$-homeomorphisms (in the precise sense of \cite[p.77]{Schikhof}).
 For this property to hold true, we need to assume that the element $2\in\k$ belongs to $\CO_\k^\times$. 
Equivalently (\cite[Proposition 2.1]{GJ2}), we need  to exclude two cases: when
 $\k$ is of characteristic two and when $\k$ is an extension of $\Q_2$. Then, $\sigma$ and $\phi$ become surjective isometries of $U_n$
 viewed as metric space for the induced ultra-metric of $\k$.
 We then  define the square root function $U_n\to U_n$, $a \mapsto a^{\frac12}$, as the 
reciprocal mapping of the square function $\sigma:U_n\to U_n$. It is then shown in \cite[Proposition 2.1]{GJ2} that  $|\sigma'|_\k=|\phi'|_\k=1$.

Combining \cite[p.287]{Schikhof} (which only treats the one-dimensional case) and \cite[p.49]{Volovich} (which only treats  the case of $\Q_p^d$), 
we deduce the following substitution formula:
\begin{equation*}
   \int_V   f(t)  \, dt = \int_U   f \circ \vf(t)  \, |{\rm Jac}_{\vf}(t)|_{\k} \, dt,
 \end{equation*}
where $U, \, V$ are compact open subsets of $\k^d$, $f : V \rightarrow \C$ is $L^1$ and 
$\vf : U \rightarrow V$ is a $C^1$-homeomorphism  such that 
${\rm Jac}_{\vf}(t) \neq 0$ for all $t \in U$. 
In particular, we will frequently use this formula   for $\vf=\sigma$ and $\vf=\phi$, which in these cases gives: 
\begin{equation}
\label{change}
 \int_{U_n}  \, f(\sigma(a))  \, da=  \int_{U_n} \,  f(a)  \, da
 \quad\mbox{and}\quad \int_{U_n}h(\phi(a))\,da=\int_{\varpi^n\CO_\k} h(x)\,dx.
 \end{equation}

With our  choice of normalization for the Haar measure of $\k$
and with the non-trivial  previously fixed standard character $\Psi$ (that is constant on $\CO_\k$ but not on $\varpi^{-1}\CO_\k$), the Fourier transform
$$
\big(\CF_\k\, f\big )(x):=\int_\k f(t)\,\Psi(xt)\,dt, \quad f\in L^1(\k)\cap L^2(\k),
$$
extends to a  unitary operator on $L^2(\k)$. Identifying a function $\tilde{f} \in L^1(\Gamma_n)$ with the function $f \in L^1(\k)$ 
invariant under  translations in $\varpi^{-n}\CO_{\k}$, we therefore have 
 \begin{equation} 
 \label{intg} 
 \sum_{[t] \in \Gamma_n}  \tilde{f}([t])    =     q^{-n} \, \int_{\k} f(t) \, dt.
\end{equation}
In particular, if $f \in L^1(\k)$ is invariant by translations in $\varpi^{-n}\CO_{\k}$ then $\CF_\k(f)$
is supported on $\varpi^n\CO_\k$ and using the identification $\widehat\Gamma_n\sim \varpi^n
\CO_\k$, we get
\begin{align}
\label{ident-fourier}
\CF_\k\, f=q^n\,\CF_{\Gamma_n}\tilde f,
\end{align}
where $$
\CF_{\Gamma_n}:L^2(\Gamma_n)\to L^2(\varpi^n\CO_\k),\quad
f\mapsto \Big[x\mapsto \sum_{[t]\in\Gamma_n} f([t]) \,\Psi(xt) \Big],
$$
is the Fourier transform on $\Gamma_n$.

We denote by $\CD(\k)$ the space of Bruhat-test functions 
  on $\k$ \cite{Bruhat}. Since $\k$ is totally disconnected, $\CD(\k)$ coincides with the space of locally
   constant compactly supported functions on $\k$. 
Recall also that the Fourier transform $\CF_\k$ is a homeomorphism of $\CD(\k)$.
   Similarly, we let $\CD(X_n)$  be
   the space of Bruhat-test functions $X_n$, that is 
 the space of functions on $X_n$ which are locally
  constant in the continuous variable $a \in U_n$ and have finite support in the discrete variable $[t]\in \Gamma_n$.
  
  \subsection{The $p$-adic Fuchs calculus.}
  
    With $\Psi$ the fixed unitary character of $({\bf  k},+)$, 
we  define the  following representation  $\pi$ of the covariance group $G_n$ on the Hilbert space $L^2(U_n)$ of square integrable functions on the configuration
space:
\begin{equation}
\label{Utheta}
\pi(a,t)\varphi(a_0) :=\Psi( a_0^{-1}at) \,
\varphi\big(a^{-1}a_0\big).
 \end{equation}
From  Mackey theory \cite{Mackey}, 
any infinite dimensional unitary irreducible representation of $G_n$ is  unitarily equivalent to a representation of this form
(and they are classified by the orbits of the dual action of $U_n$ on $\widehat\k\simeq\k$).
Let then $\Sigma$ be the group-inversion operator on $L^2(U_n)$ (i.e$.$ the antipode):
$$
\Sigma\,\vf(a):=\vf(a^{-1}).
$$
Note that $\Sigma$ is bounded since $U_n$ is unimodular. (More precisely, $\Sigma$ is unitary and self-adjoint, that is a symmetry.)
Consider then  the map 
$$
G_n \rightarrow \CB(L^2(U_n)), \, \quad g \mapsto \pi(g)\,\Sigma \,\pi^*(g).
$$
Since $\Psi|_{\CO_\k}=1$, we observe that the above map is invariant under right-translations in $H_n$. Therefore, it defines a map 
$$
\Omega: X_n=G_n/H_n \rightarrow \CB(L^2(U_n)),
$$
explicitly given by 
$$
\Omega([g]) \vf (a_0) = \Psi \big (\phi(aa_0^{-1})t \big ) \, \vf(a^2a_0^{-1}).
$$
The quantization map underlying the $p$-adic Fuchs calculus is then defined by 
$$
{\bf \Omega} : L^1(X_n) \rightarrow \CB(L^2(U_n)),\quad
f\mapsto q^n \int_{X_n} f([g]) \, \Omega([g]) \, d[g].
$$
By construction, the quantization map is $G_n$-covariant:
 \begin{equation}
  \label{equiv}
   \pi(g) \, {\bf \Omega} (f) \, \pi^*(g) = {\bf \Omega}(\lambda_{[g]}f).
     \end{equation}

\begin{rmk}{\rm
Unterberger's Fuchs calculus is defined in a similar way, with  the multiplicative group $\R^*_+$  replacing $U_n$ and with (the connected component of)
the affine group $\R^*_+\ltimes\R$ replacing both $G_n$ and  $X_n$. The common point of the two situations, is that $\R^*_+$ and $U_n$ are the largest open subgroups of 
$\R^*$ and of $\k^\times$ on which the square function is a homeomorphism.}
\end{rmk}

 \subsection{The $\star$-product and the dual $2$-cocycle.} At this stage, the most important property of the quantization map  (property requiring that $\sigma$ and $\phi$ are homeomorphisms of class $C^1$),
  is that $\bO$ extends
(from $L^1(X_n)\cap L^2(X_n)$) to a surjective isometry from $L^2(X_n)$ to the Hilbert space of Hilbert-Schmidt operators on $L^2(U_n)$
(see \cite[Proposition 3.9]{GJ2}).
We can therefore transport the associative algebra structure of Hilbert-Schmidt operators to
 $L^2(X_n)$:
$$
\star :L^2(X_n)\times L^2(X_n)\to L^2(X_n)\,,\quad (f_1,f_2)\mapsto \bO^*\big(
\bO(f_1)\,\bO(f_2)\big).
$$
It is proven in \cite[Proposition 3.13]{GJ2} that for $ f_1,f_2\in L^2(X_n)\cap L^1(X_n)$, the product $f_1\star f_2\in L^2(X_n)\cap L^\infty(X_n)$  takes the form
$$
f_1\star f_2([g_0])=\int_{X_n \times X_n}K \big ([g_0],[g_1],[g_2] \big )\,f_1([g_1])\,f_2([g_2])\, d[g_1] \,d[g_2],
$$
 where $K$ is the locally constant and bounded function on $X_n^3$ given by:
 $$
 K \big ( (a_1,[t_1]),(a_2,[t_2]),(a_3,[t_3]) \big ) = q^{2 n } \, \Psi\big(
  \phi(a_1 a_2^{-1})t_3+ \phi( a_2 a_3^{-1})t_1+\phi( a_3 a_1^{-1})t_2\big).
 $$
 By construction, the kernel $K$ is invariant under the diagonal action of $G_n$ on $X_n^3$, thus invariant under 
 the diagonal left action of $X_n$ on $X_n^3$.  Hence (see \cite[Proposition 3.14]{GJ2}) we have
  for $ f_1,f_2\in L^2(X_n)\cap L^1(X_n)$ and with $\rho$ the right regular action:
  $$
f_1\star f_2=\int_{X_n \times X_n}K \big ([e],[g_1],[g_2] \big )\,\rho_{[g_1]}(f_1)\,\rho_{[g_2]}(f_2)\, d[g_1] \,d[g_2],
$$
with
$$
K \big ([e],[g_1],[g_2] \big )= q^{2n}\,\Psi\big(-\phi(a_1)t_2+\phi(a_2)t_1\big) .
$$
Note also that $\CD(X_n)$ is stable under $\star$ (see \cite[Proposition 3.15]{GJ2}).

As explained in the introduction, the natural candidate for a unitary dual  $2$-cocycle $F$ on the group $X_n$, is  
 given by the following convolution operator, initially defined as a quadratic form on $\CD(X_n\times X_n)$:
\begin{align}
\label{FF}
F:= q^{2n} \, \int_{X_n \times X_n}  \Psi\big(\phi(a_1)t_2-\phi(a_2)t_1\big)  \,\lambda_{[g_1]^{-1}}\otimes\lambda_{[g_2]^{-1}}\,
 d[g_1]\,d[g_2].
\end{align}
The goal of the next section is to prove that $F$ preserves the Bruhat space $\CD(X_n\times X_n)$ and extends to a unitary operator on $L^2(X_n)$. 
Once this will be proven, the associativity of $ \star$ on $\CD(X_n)$ will immediately imply  the $2$-cocycle relation for $F$:
$$
(1\otimes F)\big({\rm Id}\otimes \hat\Delta(F)\big)=( F\otimes1)\big(\hat\Delta(F)\otimes{\rm Id}\big),
$$
where $ \hat\Delta$ is the coproduct of the group von Neumann algebra $W^*(X_n)$, defined on the generators by 
$ \hat\Delta(\lambda_{[g]})=\lambda_{[g]}\otimes\lambda_{[g]}$.

\section{Unitarity of the dual $2$-cocycle}
 
We start with an important factorization property:
\begin{prop}
\label{P}
Let $\Xi : U_n \times \varpi^n\CO_{\k}\times U_n \times \varpi^n\CO_{\k} \rightarrow U_n \times \varpi^n\CO_{\k}\times U_n \times \varpi^n\CO_{\k}$ 
be the continuous map defined by: 
$$
\Xi \big (a_1 ,x_1 , a_2,x_2 \big ) := \big ( \phi^{-1}(a_2^{-1}x_2)a_1, x_1,  \phi^{-1}(-a_1^{-1}x_1)a_2, x_2 \big ),
$$
and let $T_\Xi$ be the continuous operator on $\CD(U_n \times \varpi^n\CO_{\k}\times U_n \times \varpi^n\CO_{\k})$ given by:
$$
T_\Xi f=f\circ \Xi.
$$
Then, the dual $2$-cocycle $F$ factorizes as:
$$
F = \big ( {\rm Id}\otimes \CF^{-1}_{\Gamma_n} \otimes {\rm Id}\otimes \CF^{-1}_{\Gamma_n} \big ) \, T_\Xi \, \big ({\rm Id}\otimes \CF_{\Gamma_n} \otimes {\rm Id}\otimes \CF_{\Gamma_n} \big ).
$$
\end{prop}
\begin{proof}
Take $f \in \CD(X_n\times X_n)$ and let $[g_j']=(a_j',[t_j']) \in X_n$, $j=1,2$. Then we have:
\begin{align*}
 & F f ([g_1'],[g_2']) \\ & = q^{2n} \, \int_{X_n \times X_n} \Psi(-\phi(a_2)t_1) \, \Psi(\phi(a_1)t_2) \, f \big ( (a_1,[t_1]).(a_1',[t_1']) ; (a_2,[t_2]).(a_2',[t_2']) \big ) \, d[g_1]d[g_2]
 \\ & = q^{2n} \, \int_{X_n \times X_n} \Psi(-\phi(a_2)t_1) \, \Psi(\phi(a_1)t_2) \, f_1 \big ( a_1 a_1',[a_1'^{-1}t_1+t_1'] ;a_2 a_2',[a_2'^{-1}t_2+t_2'] \big ) \, d[g_1]d[g_2]. 
\end{align*}
The change of variables $[t_1] \mapsto [a_1'(t_1-t_1')]$ and $[t_2] \mapsto [a_2'(t_2-t_2')]$ leads to:
\begin{align*}
F f ([g_1'],[g_2'])& = q^{2n} \, \int_{U_n \times U_n}  \Psi(\phi(a_2)a_1't_1') \, \Psi(-\phi(a_1)a_2't_2') \\ &
\quad \times \big ({\rm Id}\otimes \CF_{\Gamma_n} \otimes {\rm Id}\otimes \CF_{\Gamma_n} \big )f
 \big ( a_1 a_1',-a_1'\phi(a_2); a_2 a_2', a_2'\phi(a_1) \big ) \, da_1da_2.
\end{align*}
The second formula in \eqref{change} then gives:
\begin{align*}
&F  f ([g_1'],[g_2'])=
 q^{2n} \, \int_{\varpi^n\CO_{\k} \times \varpi^n\CO_{\k}}  \overline{\Psi}(x_1t_1') \, \overline{\Psi}( x_2 t_2') \\ & \qquad\qquad \times 
 \big ({\rm Id}\otimes \CF_{\Gamma_n} \otimes {\rm Id}\otimes \CF_{\Gamma_n} \big ) f \big ( \phi^{-1}(a_2'^{-1}x_2)a_1', x_1;  \phi^{-1}(-a_1'^{-1}x_1)a_2', x_2 \big ) \, dx_1dx_2,
\end{align*}
which is the announced formula.
\end{proof}
Since the operator of composition by a continuous function preserves the space of locally constant functions, we deduce:
\begin{corr}
The dual $2$-cocycle $F$ is continuous on the Bruhat space $\CD(X_n\times X_n)$. In particular,  $F$ also makes sense as a densely defined
operator affiliated with the group von Neumann algebra $W^*(X_n\times X_n)$.
\end{corr}

Next, we come to our main technical result:

\begin{lem}
\label{lemme}
The map $\Xi $ defined in Proposition \ref{P} is a homeomorphism of class $C^1$   with
$|{\rm Jac}_{ \Xi} |_\k=1$.
\end{lem}

\begin{proof}
The main difficulty is to  show that $\Xi$ is bijective. Due to the specific form of the map $\Xi$, it is equivalent  to prove that for any 
$(a_1',x_1',a_2',x_2')\in U_n \times \varpi^n\CO_{\k}\times U_n \times \varpi^n\CO_{\k}$, there exists a unique  
$(a_1,a_2)\in U_n \times U_n$  such that
\begin{align}
\label{XI}
\Xi \big (a_1,x_1',a_2,x_2' \big ) =   \big (a_1',x_1',a_2',x_2' \big ).
\end{align}
The relation \eqref{XI} give the (equivalent)  systems of equations:
\begin{equation*}
\begin{cases}
    \phi^{-1}(a_2^{-1}x_2')a_1 = a_1'    \\
    \phi^{-1}(-a_1^{-1}x_1')a_2 = a_2'  
\end{cases}
\hspace{-0.1cm}
\Leftrightarrow \,\,
\begin{cases}
     a_2^{-1}x'_2 =  \phi(a_1'a_1^{-1})    \\
     -a_1^{-1}x'_1=   \phi(a_2'a_2^{-1})  
\end{cases}
\hspace{-0.1cm}
\Leftrightarrow \,\,
\begin{cases}
a_1'^{-1}a_1^2+a_2^{-1}x_2' a_1-a_1'=0\\
a_2'^{-1}a_2^2-a_1^{-1}x_1' a_2-a_2'=0
\end{cases}\hspace{-0.4cm}.
\end{equation*}
Note first that this system is invariant under the dilations:
$$
(a_1,a_2,;a'_1,a'_2,x'_1,x'_2)\mapsto (\alpha_1a_1,\alpha_2a_2,;\alpha_1a'_1,\alpha_2a'_2,\alpha_1x'_1,\alpha_2x'_2),\quad
\alpha_1,\alpha_2\in U_n.
$$
Hence,  we get the equivalent system
of coupled quadratic equations in term of the homogeneous variables $u_j:=a_j'^{-1}a_j\in U_n$ and  parameters $X_j:=a_j'^{-1}x_j\in\varpi^n\CO_k$, $j=1,2$:
\begin{equation}
\label{Sy}
\begin{cases}
u_1^2+X_2\,u_2^{-1}\,u_1-1=0\\
u_2^2-X_1\,u_1^{-1} \,u_2-1=0
\end{cases}\hspace{-0.4cm}.
\end{equation}
Note also that this system
is invariant under the transformation  
\begin{align}
\label{inv}
(u_1,u_2,X_1,X_2)\mapsto(u_2,u_1,-X_2,-X_1).
\end{align}

The discriminant of the first equation in \eqref{Sy} (treating $u_2\in U_n$ as a parameter) is given by
$$
\Delta_1 :=4+ \frac{X_2^2}{u_2^{2}}=4\Big(1+ \frac{X_2^2}{4u_2^{2}}\Big).
$$
Since   $X_2\in \varpi^n\CO_\k$, $u_2\in U_n\subset \CO_\k^\times$ and since $4$ is a unit too, we have:
$$
1+ \frac{X_2^2}{4u_2^{2}}\in 1+ \varpi^{2n}\CO_\k\subset  1+ \varpi^{n}\CO_\k=U_n,
$$
which therefore possesses a (unique) square root  in $U_n$. Since $4$ possesses exactly two square roots in $\k$ (which are 
$\pm 2$), $\Delta_1$  also possesses exactly two square roots in $\k$. Hence, the first equation of the quadratic system 
possesses exactly two solutions in $\k$ namely: 
$$
Y_\pm=\pm   \Big (1+  \frac{X_2^2}{4u_2^{2}} \Big )^{\frac{1}{2}}-\frac{X_2}{2u_2}\in \pm U_n+ \varpi^{n}\CO_\k=\pm1+\varpi^{n}\CO_\k.
$$
The solution $Y_-$ is not admissible because it does not belong to $U_n$ as it should be. Hence, we get only one solution, namely
\begin{equation}
\label{eq2}
u_1 =  \Big (1+  \frac{X_2^2}{4u_2^{2}} \Big )^{\frac{1}{2}}-\frac{X_2}{2u_2} \in U_n .  
\end{equation}
Substituting this expression for $u_1$ in the second equation of the system \eqref{Sy} yields the relation: 
\begin{align*}
 X_1 = \Big ( \Big (1+  \frac{X_2^2}{4u_2^{2}} \Big )^{\frac{1}{2}}-\frac{X_2}{2u_2}\Big)(u_2-u_2^{-1} ),
\end{align*}
which is equivalent to:
\begin{equation}
\label{eq3}
 X_1 +\frac{X_2}{2u_2}(u_2-u_2^{-1} )= \Big (1+  \frac{X_2^2}{4u_2^{2}} \Big )^{\frac{1}{2}}(u_2-u_2^{-1} ).
 \end{equation}
Squaring the relation \eqref{eq3} and multiplying it by $u_2^2$  give,
in term of the variable  $U:=u_2^2\in U_n$,  the following quadratic equation:
\begin{align}
\label{X}
U^2 - \big (2+ X_1X_2 +X_1^2   \big ) U + \big ( 1+ X_1X_2 \big )    =   0. 
\end{align}
The discriminant of  equation \eqref{X} is:
\begin{align*}
 \Delta_2 = X_1^2 \big (  4 + (X_1 +X_2)^2 \big ).
\end{align*}
Since 
$$
1 + \Big(\frac{X_1}2 +\frac{X_2}2\Big)^2\in 1+ \varpi^{2n}\CO_\k \subset 1+ \varpi^{n}\CO_\k=U_n,
$$
 $ \Delta_2$ possesses exactly two square roots in $\k$, given by
$$
\pm2X_1 \Big(1 + \Big(\frac{X_1}2 +\frac{X_2}2\Big)^2\Big)^{\frac12}\in  \varpi^{n}\CO_\k,
$$  
and thus, the equation \eqref{X} possesses exactly two solutions:
\begin{equation}
\label{eq4}
U_\pm =1+ \frac{X_1X_2 }2+\frac{X_1^2}2\pm X_1 \Big(1 + \Big(\frac{X_1}2 +\frac{X_2}2\Big)^2\Big)^{\frac12}.
\end{equation}
We need to prove that only one solution is admissible and this not that easy since here both $U_+$ and $U_-$ belong to $U_n$. To this aim,
note first that we can assume without loss of generality that $X_1\ne 0$ (if $X_1=0$ there is no ambiguity).
We then evaluate the relation \eqref{eq3} in $U_\pm$  and multiply it by $U_\pm$ to get:
\begin{align*}
 X_1U_\pm=\Big(\Big(U_\pm+\frac{X_2^2}4\Big)^{\frac12}-\frac{X_2}2\Big)(U_\pm-1).
\end{align*}
Since
$$
U_\pm -1= \frac{X_1X_2 }2+\frac{X_1^2}2\pm X_1 \Big(1 + \Big(\frac{X_1}2 +\frac{X_2}2\Big)^2\Big)^{\frac12},
$$
we get, after simplification by $X_1\ne 0$, the condition:
$$
U_\pm=\Big(\Big(U_\pm+\frac{X_2^2}4\Big)^{\frac12}-\frac{X_2}2\Big)\Big(\frac{X_2 }2+\frac{X_1}2\pm  \Big(1 + \Big(\frac{X_1}2 +\frac{X_2}2\Big)^2\Big)^{\frac12}\Big).
$$
The first factor  in the above expression belongs to $U_n$ while the second factor 
 belongs to $\pm1+\varpi^n\CO_\k$.
  This condition  thus excludes $U_-$. Hence, $u_2^2=U_+$, which only gives  one solution since the square root function is bijective
 on $U_n$:
 $$
 u_2=\Big(1+ \frac{X_1X_2 }2+\frac{X_1^2}2+ X_1 \Big(1 + \Big(\frac{X_1}2 +\frac{X_2}2\Big)^2\Big)^{\frac12}\Big)^{\frac12}.
 $$
 By the symmetry \eqref{inv} of the system, we also get:
  $$
 u_1=\Big(1+ \frac{X_1X_2 }2+\frac{X_2^2}2- X_2 \Big(1 + \Big(\frac{X_1}2 +\frac{X_2}2\Big)^2\Big)^{\frac12}\Big)^{\frac12}.
 $$
  Finally, in term of the original variables and parameters, we get: 
\begin{align*}
a_2 &= a_2' \Big ( 1+ \frac{x_1'x_2'}{2a_1'a_2'} + \frac{x_1'^2}{2a_1'^{2}}  +  \frac{x_1'}{ a_1'} 
\Big (  1 + \Big( \frac{x_1'}{2a_1'}+\frac{x_2'}{2a_2'}\Big)^2 \Big )^{ \frac{1}{2}  } \Big )^{\frac{1}{2}},\\
a_1 &= a_1' \Big ( 1+ \frac{x_1'x_2'}{2a_1'a_2'} + \frac{x_2'^2}{2a_2'^{2}}  -  \frac{x_2'}{ a_2'} 
\Big (  1 + \Big( \frac{x_1'}{2a_1'}+\frac{x_2'}{2a_2'}\Big)^2 \Big )^{ \frac{1}{2}  } \Big )^{\frac{1}{2}}.
\end{align*}
This shows that $\Xi$ is one-to-one and onto and the explicit expressions above entail that it is moreover a homeomorphism
of class $C^1$. 

It remains to compute the absolute value of the Jacobian of $\Xi$.
Since for every $a \in U_n$, $\phi'(a) = 1+a^{-2}$, we have 
$$
(\phi^{-1})'(x) = \Big ( 1 + \big ( \phi^{-1}(x) \big )^{-2} \Big )^{-1}, \quad \forall x \in \varpi^n\CO_{\k}.
$$
As a consequence, we get 
\begin{align*}
  {\rm Jac}_{\Xi} (a_1,x_1,a_2,x_2)   &  =    \phi^{-1}\big (-a_1^{-1}x_1 \big ) \,  \phi^{-1}\big ( a_2^{-1}x_2 \big ) \\
  +&  a_1^{-1}a_2^{-1} \Big ( 1+ \big (  \phi^{-1}\big (-a_1^{-1}x_1 \big ) \big )^{-2}     \Big )^{-1} \Big ( 1+  \big (\phi^{-1}\big ( a_2^{-1}x_2 \big ) \big )^{-2}   \Big )^{-1} 
   x_1x_2.
\end{align*}
The first term clearly belongs  to $U_n$. For the second term, note first that
$$
1+ \big (  \phi^{-1}\big (\pm a_j^{-1}x_j \big ) \big )^{-2}   \in 1+U_n=2+\varpi^n\CO_\k\subset \CO_\k^\times.
$$
Hence in the second term, all the factors preceding  $ x_1x_2$  are units. Hence the second term belongs to $\varpi^{2n}\CO_\k$.
This proves that $ {\rm Jac}_{\Xi} $ takes values in $U_n+\varpi^{2n}\CO_\k=U_n$, from which the claim follows.
\end{proof}

An immediate consequence of Lemma \ref{lemme} is that the operator $T_\Xi$ is unitary on 
$L^2(U_n \times \varpi^n\CO_{\k}\times U_n \times \varpi^n\CO_{\k})$. Since the partial Fourier
transform $ {\rm Id}\otimes \CF_{\Gamma_n}$ is unitary from $L^2(X_n)$ to $L^2(U_n \times \varpi^n\CO_{\k})$, 
Proposition \ref{P} immediately implies:

\begin{thm}
\label{Th}
Let $\k$ be a local field which is not of characteristic $2$, nor an extension of $\Q_2$. Then,  
 the dual $2$-cocycle $F$ given in \eqref{FF} is a  unitary operator  on $L^2(X_n\times X_n)$.
\end{thm}

\section*{Acknowledgments}
This work is part of my PhD thesis. I would like to warmly thank my advisor Victor Gayral, for his extreme patience as well as his unshakeable support.

\appendix
\section{Proof of $2$-cocyclicity}
Our goal here is to provide a direct proof  that the (unitary) convolution operator:
\begin{align*}
F= q^{2n} \, \int_{X_n \times X_n}  \Psi\big(\phi(a_1)t_2-\phi(a_2)t_1\big)  \,\lambda_{[g_1]^{-1}}\otimes\lambda_{[g_2]^{-1}}\,
 d[g_1]\,d[g_2]
\end{align*}
satisfies the $2$-cocycle relation:
$$
(1\otimes F)\big({\rm Id}\otimes \hat\Delta(F)\big)=( F\otimes1)\big(\hat\Delta(F)\otimes{\rm Id}\big).
$$

For convenience, we shall work with the adjoint of  $F$, given  on $\CD(X_n\times X_n)$ by:
\begin{align*}
F^*:= q^{2n} \, \int_{X_n ^2}  \overline{\Psi}\big(\phi(a_1)t_2-\phi(a_2)t_1\big)  \,\lambda_{[g_1]}\otimes\lambda_{[g_2]}\,
 d[g_1]\,d[g_2].
\end{align*}
In term of $F^*$,   the cocycle relation reads:
$$
\big({\rm Id}\otimes \hat\Delta(F^*)\big) (1\otimes F^*) =\big(\hat\Delta(F^*)\otimes{\rm Id}\big)( F^*\otimes1).
$$
By density of $\CD(X_n\times X_n\times X_n)$ in $L^2(X_n\times X_n\times X_n)$, it is enough to prove this relation on test  functions.
For $f \in \CD(X_n\times X_n\times X_n)$ and $[g],[g'],[g''] \in X_n$, we have:
\begin{align*} 
&\big({\rm Id}\otimes \hat\Delta(F^*)\big)f([g],[g'],[g'']) 
= q^{2n} \, \int_{X_n^2} \Psi(\phi(a_2)t_1) \, \overline{\Psi}(\phi(a_1)t_2) \\
&\quad\quad\qquad\qquad\qquad\times f \big (a a_1^{-1},[t- a_1a^{-1}t_1] ;  a'a_2^{-1},[t'- a_2 a'^{-1}t_2]; a'' a_2^{-1},[t''-a_2 a''^{-1}t_2] \big ) \, d[g_1] d[g_2],\\ 
&\big(\hat\Delta(F^*)\otimes{\rm Id}\big)f([g],[g'],[g'']) 
= q^{2n} \, \int_{X_n^2} \Psi(\phi(a_2)t_1) \, \overline{\Psi}(\phi(a_1)t_2) \\
&\quad\quad\qquad\qquad\qquad\times f \big (a a_1^{-1},[t- a_1a^{-1}t_1] ;  a'a_1^{-1},[t'- a_1 a'^{-1}t_1]; a'' a_2^{-1},[t''-a_2 a''^{-1}t_2] \big ) \, d[g_1] d[g_2]. 
\end{align*}
 
Since $f$ is compactly supported,   the  integrals above are, for fixed $[g],[g'],[g''] \in X_n$, absolutely convergent.
To simplify a little the notations, we shall let $\CF_{\Gamma_n}$ be the partial Fourier transform:
\begin{align*}
\CF_{\Gamma_n}&: L^2(U_n\times \Gamma_n)\to L^2(U_n\times \widehat\Gamma_n)\simeq L^2(U_n\times \varpi^{-n}\CO_\k).
 \end{align*}
The dilation $t_1 \mapsto -aa_1^{-1}t_1$ followed by the translation $t_1 \mapsto t_1 - t$ in the first expression and 
the dilation $t_2 \mapsto -a''a_2^{-1}t_2$ followed by the translation $t_2 \mapsto t_2 - t''$ in the second, entail:
\begin{align*}
&\big({\rm Id}\otimes \hat\Delta(F^*)\big)f([g],[g'],[g'']) 
=
q^{2n} \, \int_{U_n \times X_n}  \Psi(aa_1^{-1}\phi(a_2)t) \, \overline{\Psi}(\phi(a_1)t_2) \\ 
&\qquad\quad \times \big ( \CF_{\Gamma_n} \otimes 1 \otimes 1  \big ) f \big (aa_1^{-1}, - a a_1 ^{-1} \phi(a_2) ; a' a_2^{-1},[t'-a_2a'^{-1}t_2];a''a_2^{-1},[t''-a_2a''^{-1}t_2] \big ) \, da_1d[g_2],\\
&\big(\hat\Delta(F^*)\otimes{\rm Id}\big)f([g],[g'],[g'']) 
=
q^{2n} \, \int_{U_n \times X_n}  \overline{\Psi}(a''a_2^{-1} \phi(a_1)t'')\, \Psi(\phi(a_2)t_1)  \\ 
&\qquad\quad \times \big ( 1 \otimes 1 \otimes \CF_{\Gamma_n} \big ) f \big (a a_1^{-1},[t- a_1a^{-1}t_1]  ; a' a_1^{-1},[t'-a_1a'^{-1}t_1];a''a_2^{-1},a''a_2^{-1}\phi(a_1) \big ) \, da_2
d[g_1].
 \end{align*}

  For $a,a',a'',a_1,a_2\in U_n$ and $[t],[t'],[t'']\in\Gamma_n$ fixed, we shall consider the functions
\begin{align*}
 &\vf_{a,a',a'',a_1,a_2,[t'],[t'']}([t_2]):=\\
 & \qquad\quad\qquad\qquad\big ( \CF_{\Gamma_n} \otimes 1 \otimes 1  \big ) f \big ( a a_1^{-1}, - a a_1^{-1}\phi(a_2) ; a'a_2^{-1},[t'-a_2a'^{-1}t_2]; a'' a_2^{-1},[t''- a_2 a''^{-1}t_2] \big ),
 \\
 &\psi_{a,a',a'',a_1,a_2,[t],[t']}([t_1]):=\\
 & \qquad\quad\qquad\qquad\big ( 1 \otimes 1 \otimes \CF_{\Gamma_n} \big ) f \big (a a_1^{-1},[t- a_1a^{-1}t_1]  ; a' a_1^{-1},[t'-a_1a'^{-1}t_1];a''a_2^{-1},a''a_2^{-1}\phi(a_1) \big ).
 \end{align*}
In term of these functions, we have 
 \begin{align}
\label{form}
&\big({\rm Id}\otimes \hat\Delta(F^*)\big)f([g],[g'],[g'']) = q^{2n} \, \int_{U_n^2}  \Psi(aa_1^{-1}\phi(a_2) t) \,\big( \CF_{\Gamma_n}\vf_{a,a',a'',a_1,a_2,[t'],[t'']}\big)
(-\phi(a_1)) \, da_1da_2,
 \\
\label{formz}
&\big(\hat\Delta(F^*)\otimes{\rm Id}\big)f([g],[g'],[g'']) = q^{2n} \, \int_{U_n ^2}  \overline{\Psi}(a''a_2^{-1} \phi(a_1)t'') \,\big( \CF_{\Gamma_n}\psi_{a,a',a'',a_1,a_2,[t],[t']}\big)
(\phi(a_2)) \, da_1da_2.
 \end{align}

In order to compute $\big({\rm Id}\otimes \hat\Delta(F^*)\big) (1\otimes F^*)f$ and $\big(\hat\Delta(F^*)\otimes{\rm Id}\big)( F^*\otimes1)f$, we will substitute $f$ by $(1\otimes F^*)f$
in   \eqref{form} and
$f$ by $(F^*\otimes 1)f$
in   \eqref{formz}. To do that, we first need a convenient formula for  $F^*h$, with $h\in\CD(X_n\times X_n)$. Using (partially) the factorization 
given in Proposition \eqref{P},  we get:
\begin{align*}
 F^*h([g_1],[g_2]) &= q^{2n} \, \int_{U_n^2}  \Psi( a_1 a_3^{-1}\phi(a_4)t_1) \, \overline{\Psi}( a_2a_4^{-1}\phi(a_3) t_2) \\&
\qquad\qquad\qquad \times (\CF_{\Gamma_n} \otimes \CF_{\Gamma_n}) h\big ( a_1 a_3^{-1},- a_1 a_3^{-1}\phi(a_4) ;  a_2 a_4^{-1},  a_2a_4^{-1}\phi(a_3) \big )  \, da_3da_4 .
\end{align*}
From this, we deduce:
\begin{align}
\label{trois}
 & \big ( \CF_{\Gamma_n} \otimes 1 \otimes 1  \big ) (1 \otimes F^*) f\big( a a_1^{-1}, - a a_1^{-1}\phi(a_2) ; a'a_2^{-1},[t'-a_2a'^{-1}t_2]; a'' a_2^{-1},[t''- a_2 a''^{-1}t_2]\big)   \\ 
&\qquad\quad=  q^{2n} \, \int_{U_n ^2} \,  \Psi \big ( a' a_2^{-1}a_3^{-1} \phi(a_4)  (t'- a_2a'^{-1}t_2) \big ) \, \overline{\Psi} \big (a''a_2^{-1}a_4^{-1}\phi(a_3)(t''-a_2a''^{-1}t_2) \big ) \nonumber\\
 &\qquad\qquad\qquad\qquad\times  (\CF_{\Gamma_n} \otimes  \CF_{\Gamma_n} \otimes  \CF_{\Gamma_n}) f \big ( aa_1^{-1}, -aa_1^{-1}\phi(a_2)  ;  a' a_2^{-1}a_3^{-1},
 - a'a_2^{-1}a_3^{-1}\phi(a_4) ;
 \nonumber\\
 &\hspace{8.5cm} a''a_2^{-1}a_4^{-1},  a''a_2^{-1}a_4^{-1}\phi(a_3)  \big )\, da_3da_4,\nonumber\\
\label{troisbis}
& \big ( 1\otimes 1 \otimes \CF_{\Gamma_n}   \big ) (F^* \otimes 1) f \big (a a_1^{-1},[t- a_1a^{-1}t_1]  ; a' a_1^{-1},[t'-a_1a'^{-1}t_1];a''a_2^{-1},a_2^{-1}a''\phi(a_1) \big )  \\ 
&\qquad\quad=  q^{2n} \, \int_{U_n ^2} \,  \Psi \big ( a a_1^{-1}a_3^{-1} \phi(a_4)  (t- a_1a^{-1}t_1) \big ) \,
 \overline{\Psi} \big (a'a_1^{-1}a_4^{-1}\phi(a_3)(t'-a_1a'^{-1}t_1) \big ) \nonumber\\
 &\qquad\qquad\times  
 (\CF_{\Gamma_n} \otimes  \CF_{\Gamma_n} \otimes  \CF_{\Gamma_n}) f \big ( aa_1^{-1}a_3^{-1}, -aa_1^{-1}a_3^{-1}\phi(a_4)  
 ;  a' a_1^{-1}a_4^{-1}, a'a_1^{-1}a_4^{-1}\phi(a_3) ;
 \nonumber\\
 &\hspace{8.9cm} a''a_2^{-1},  a''a_2^{-1}\phi(a_1)  \big )\, da_3da_4.\nonumber
\end{align}
Using that $\phi(a_3) a_4^{-1}- \phi(a_4) a_3^{-1}  = \phi(a_3a_4^{-1}),$  \eqref{trois} and \eqref{troisbis} are respectively
given by:
\begin{align*}
&  q^{2n} \, \int_{U_n ^2}  \Psi \big (\phi(a_3 a_4^{-1}) t_2 \big ) \, \Psi \big (  a' a_2^{-1} a_3^{-1} \phi(a_4)   t' \big ) \, 
\overline{\Psi} \big ( a''a_2^{-1} a_4^{-1} \phi(a_3) t'' \big )  
  (\CF_{\Gamma_n} \otimes  \CF_{\Gamma_n} \otimes  \CF_{\Gamma_n}) f \\
  &\qquad\qquad\big ( aa_1^{-1}, -aa_1^{-1}\phi(a_2) ;  a' a_2^{-1}a_3^{-1},- a'a_2^{-1}a_3^{-1} \phi(a_4);a''a_2^{-1}a_4^{-1},  a'' a_2^{-1}a_4^{-1}\phi(a_3) \big ) \, da_3da_4,
\end{align*}
and by
 \begin{align*}
&  q^{2n} \, \int_{U_n ^2}   \Psi\big(\phi(a_3a_4^{-1})t_1 \big ) \,  \Psi \big ( a a_1^{-1} a_3^{-1} \phi(a_4) t\big)
\,
 \overline{\Psi} \big (a'a_1^{-1}a_4^{-1}\phi(a_3)t'\big)\,
  (\CF_{\Gamma_n} \otimes  \CF_{\Gamma_n} \otimes  \CF_{\Gamma_n}) f \\
 &\qquad\qquad\big ( aa_1^{-1}a_3^{-1}, -aa_1^{-1}a_3^{-1}\phi(a_4)  
 ;  a' a_1^{-1}a_4^{-1}, a'a_1^{-1}a_4^{-1}\phi(a_3) ;
 a''a_2^{-1},  a''a_2^{-1}\phi(a_1)  \big )\, da_3da_4.
\end{align*}

By the dilation $a_3 \mapsto a_4a_3$ and the change of variable $x_3 = \phi(a_3)$, \eqref{trois} and \eqref{troisbis}  become:
\begin{align*}
&   q^{2n} \, \int_{\varpi^n\CO_{\k} \times U_n}  \Psi \big (x_3 t_2 \big ) \, \Psi \big (  a' a_2^{-1}  a_4^{-1} \phi(a_4) (\phi^{-1}(x_3))^{-1} t' \big ) \, 
\overline{\Psi} \big ( a''a_2^{-1}a_4^{-1} \phi(\phi^{-1}(x_3) a_4)  t'' \big )  \\
& \times  (\CF_{\Gamma_n} \otimes  \CF_{\Gamma_n} \otimes  \CF_{\Gamma_n}) f \big ( aa_1^{-1}, -aa_1^{-1}\phi(a_2)   ;  a' a_2^{-1}a_4^{-1}(\phi^{-1}(x_3))^{-1},
- a'a_2^{-1}a_4^{-1}\phi(a_4)(\phi^{-1}(x_3))^{-1}; \\
 &\hspace{8.5cm} a''a_2^{-1}a_4^{-1},  a''a_2^{-1}a_4^{-1}\phi(\phi^{-1}(x_3)a_4)  \big) \, dx_3da_4,
\end{align*}
and
\begin{align*}
&  q^{2n} \, \int_{\varpi^n\CO_\k \times U_n}  \Psi\big(x_3t_1 \big )  \, \Psi \big ( a a_1^{-1}a_4^{-1} \phi(a_4)( \phi^{-1}(x_3))^{-1} t\big)
\,
 \overline{\Psi} \big (a'a_1^{-1}a_4^{-1}\phi(\phi^{-1}(x_3)a_4)t'\big)
 \\
& \times  (\CF_{\Gamma_n} \otimes  \CF_{\Gamma_n} \otimes  \CF_{\Gamma_n}) f \big ( aa_1^{-1}a_4^{-1}(\phi^{-1}(x_3))^{-1}, -aa_1^{-1}a_4^{-1}\phi(a_4)(\phi^{-1}(x_3))^{-1}  
 ;  \nonumber\\
 &\hspace{3cm}a' a_1^{-1}a_4^{-1}, a'a_1^{-1}a_4^{-1}\phi(\phi^{-1}(x_3)a_4) ;
  a''a_2^{-1},  a''a_2^{-1} \phi(a_1) \big )\, dx_3da_4.
\end{align*}

Let us  introduce, for fixed $a,a',a'',a_1,a_2,a_4\in U_n$ and $[t],[t'],[t'']\in\Gamma_n$, the following functions
\begin{align*}
&\tilde\vf_{a,a',a'',a_1,a_2,a_4,[t'],[t'']}(x_3):= \Psi \big (  a' a_2^{-1}  a_4^{-1} \phi(a_4) (\phi^{-1}(x_3))^{-1} t' \big ) \, 
\overline{\Psi} \big ( a''a_2^{-1}a_4^{-1} \phi(\phi^{-1}(x_3) a_4)  t'' \big )  \\
& \times  (\CF_{\Gamma_n} \otimes  \CF_{\Gamma_n} \otimes  \CF_{\Gamma_n}) f \big ( aa_1^{-1}, -aa_1^{-1}\phi(a_2)   ;  a' a_2^{-1}a_4^{-1}(\phi^{-1}(x_3))^{-1},
- a'a_2^{-1}a_4^{-1}\phi(a_4)(\phi^{-1}(x_3))^{-1}; \\
 &\hspace{9cm} a''a_2^{-1}a_4^{-1},  a''a_2^{-1}a_4^{-1}\phi(\phi^{-1}(x_3)a_4)  \big) ,\\
&\tilde\psi_{a,a',a'',a_1,a_2,a_4,[t],[t']}(x_3):=\Psi \big ( a a_1^{-1}a_4^{-1} \phi(a_4)( \phi^{-1}(x_3))^{-1} t\big)
\,
 \overline{\Psi} \big (a'a_1^{-1}a_4^{-1}\phi(\phi^{-1}(x_3)a_4)t'\big)\\
& \times  (\CF_{\Gamma_n} \otimes  \CF_{\Gamma_n} \otimes  \CF_{\Gamma_n}) f \big ( aa_1^{-1}a_4^{-1}(\phi^{-1}(x_3))^{-1}, -aa_1^{-1}a_4^{-1}\phi(a_4)(\phi^{-1}(x_3))^{-1}  
 ;  \nonumber\\
 &\hspace{6cm}a' a_1^{-1}a_4^{-1}, a'a_1^{-1}a_4^{-1}\phi(\phi^{-1}(x_3)a_4) ;
  a''a_2^{-1},  a''a_2^{-1} \phi(a_1) \big ).
\end{align*}
Using the isomorphism $\widehat\Gamma_n\sim \varpi^{n}\CO_\k$, we then see that the integral over $x_3\in\varpi^n\CO_\k$ can be rewritten as an inverse Fourier
transform, and we get:
\begin{align*}
 & \big ( \CF_{\Gamma_n} \otimes 1 \otimes 1  \big ) (1 \otimes F^*) f\big( a a_1^{-1}, - a a_1^{-1}\phi(a_2) ; a'a_2^{-1},[t'-a_2a'^{-1}t_2]; a'' a_2^{-1},[t''- a_2 a''^{-1}t_2]\big)   \\ 
&\hspace{8cm}= \int_{U_n}\CF_{\Gamma_n}^{-1}\tilde\vf_{a,a',a'',a_1,a_2,a_4,[t'],[t'']}([-t_2])\,da_4,\\
& \big ( 1\otimes 1 \otimes \CF_{\Gamma_n}   \big ) (F^* \otimes 1) f \big (a a_1^{-1},[t- a_1a^{-1}t_1]  ; a' a_1^{-1},[t'-a_1a'^{-1}t_1];a''a_2^{-1},a_2^{-1}a''\phi(a_1) \big )  \\   
&\hspace{8cm}= \int_{U_n}\CF_{\Gamma_n}^{-1}\tilde\psi_{a,a',a'',a_1,a_2,a_4,[t],[t']}([-t_1])\,da_4.
\end{align*}
 Plugging these expressions in \eqref{form} and \eqref{formz}, we get after the cancellation of the Fourier  transforms:
\begin{align*}
& \big({\rm Id}\otimes \hat\Delta(F^*)\big) (1\otimes F^*)f([g],[g'],[g'']) \\
&=  \int_{ U_n^3}    \Psi \big ( a a_1^{-1} \phi(a_2) t \big ) \,  \Psi \big ( a' a_1^{-1}a_2^{-1}     a_4^{-1} \phi(a_4) t' \big ) \, \overline\Psi\big ( a'' a_2^{-1}   a_4^{-1} \phi(a_1a_4) t'')   (\CF_{\Gamma_n} \otimes \CF_{\Gamma_n} \otimes \CF_{\Gamma_n}) f \\ &
\qquad \big (  a a_1^{-1},  -a a_1^{-1}\phi(a_2)  ; a' a_1^{-1}a_2^{-1}a_4^{-1}, -a'a_1^{-1}a_2^{-1}a_4^{-1} \phi(a_4) ;  a'' a_2^{-1}a_4^{-1}, a'' a_2^{-1}a_4^{-1}\phi(a_1a_4) \big) \,  da_1da_2da_4 . 
\end{align*}
and
\begin{align*}
&\big(\hat\Delta(F^*)\otimes{\rm Id}\big)( F^*\otimes1)f([g],[g'],[g''])  \\ & =  \int_{U_n^3}    \Psi \big (  aa_2a_1^{-1} a_4^{-1}\phi(a_4)  t \big )   \,  
\overline\Psi \big ( a'a_1^{-1}a_4^{-1}\phi(a_4a_2^{-1})t' \big )  \, \overline \Psi\big (a''a_2^{-1}  \phi(a_1)t'' \big )  (\CF_{\Gamma_n} \otimes \CF_{\Gamma_n} \otimes \CF_{\Gamma_n}) f
\\
&\qquad \big(aa_2a_1^{-1} a_4^{-1}, aa_2a_1^{-1} a_4^{-1}\phi(a_4^{-1}), a'a_1^{-1}a_4^{-1}, a'a_1^{-1}a_4^{-1} \phi(a_4 a_2^{-1});  a''a_2^{-1},  a''a_2^{-1}\phi(a_1) \big )  \, 
da_1da_2da_4.
\end{align*}

The last point is to show that the two above expressions are actually the same.
Performing the dilations $a_2 \mapsto a_4 a_2$ and  $a_1 \mapsto a_2a_1$, the last expression becomes:
\begin{align*}
& \big(\hat\Delta(F^*)\otimes{\rm Id}\big)( F^*\otimes1)f([g],[g'],[g'']) \\ &  
 =  \int_{U_n^3}    \Psi \big ( aa_1^{-1}\phi(a_4) t \big )  \,  \Psi \big ( a'a_1^{-1}a_2^{-1}a_4^{-1}\phi( a_2 )  t' \big ) \, 
  \overline\Psi \big (a'' a_2^{-1}a_4^{-1}  \phi(a_1a_2)t'' \big )  (\CF_{\Gamma_n} \otimes \CF_{\Gamma_n} \otimes \CF_{\Gamma_n}) f
  \\ & \big (aa_1^{-1}, -aa_1^{-1}\phi(a_4) ,  
 a' a_1^{-1}a_2^{-1}a_4^{-1},- a'a_1^{-1}a_2^{-1}a_4^{-1}\phi(a_2)  ; a''a_2^{-1}a_4^{-1}, a''a_2^{-1}a_4^{-1}  \phi(a_1a_2) \big ) \,  da_1da_2 da_4,
\end{align*}
which, after the relabelling $a_2\leftrightarrow a_4$, is exactly  the formula we found for  $\big({\rm Id}\otimes \hat\Delta(F^*)\big) (1\otimes F^*)f$.

\end{document}